\documentclass[12pt]{amsart}
\usepackage[cp1251]{inputenc}
\usepackage[T2A]{fontenc}
\usepackage[english]{babel}
\usepackage{amsmath,amsfonts,amssymb}
\usepackage{geometry}
\usepackage{amsmath, amsthm, amscd, amsfonts, amssymb, graphicx, color}
\usepackage[bookmarksnumbered, colorlinks, plainpages]{hyperref}

\newtheorem{theorem}{Theorem}

\newtheorem{lemma}{Lemma}
\newtheorem{theoremapp}{Theorem}

\theoremstyle{remark}
\newtheorem{remark}{Remark}

\theoremstyle{definition}

\begin{document}

\title[Localization principles for Schr\"odinger operator]
{Localization principles for Schr\"odinger operator with a singular matrix potential}


\author[V. Mikhailets]{Vladimir Mikhailets}
\address{Institute of Mathematics, National Academy of Sciences of Ukraine, Tere\-shchen\-kiv\-ska Str. 3, 01004 Kyiv-4, Ukraine; \newline
National Technical University of Ukraine\ "Kyiv Polytechnic Institute", Peremohy Avenue 37, 03056, Kyiv-56, Ukraine Ukraine}
\email{mikhailets@imath.kiev.ua}


\author[A. Murach]{Aleksandr Murach}
\address{Institute of Mathematics, National Academy of Sciences of Ukraine, Tereshchenkivska Str. 3, 01004 Kyiv-4, Ukraine}
\email{murach@imath.kiev.ua}


\author[V. Novikov]{Viktor Novikov}
\address{Institute of Mathematics, National Academy of Sciences of Ukraine, Tereshchenkivska Str. 3, 01004 Kyiv-4, Ukraine}
\email{thesuperpothead@gmail.com}



\begin{abstract}
We study the spectrum of the one-dimensional Schr\"{o}dinger operator $H_0$ with a matrix singular distributional potential $q=Q'$ where $Q\in L^{2}_{\mathrm{loc}}(\mathbb{R},\mathbb{C}^{m})$. We obtain generalizations of Ismagilov's localization principles, which give necessary and sufficient conditions for the spectrum of $H_0$ to be bounded below and discrete.
\end{abstract}

\maketitle

\section{Introduction}\label{sect1}

Schr\"odinger operators occupy a special position in the modern mathematical physics because they have numerous applications to physical problems and other branches of mathematics; see, e.g., \cite{CFKS}. Nowadays the spectral theory of these operators has developed very profoundly and contains a number of fundamental results. Specifically, this concerns the questions about self-adjointness, semiboundedness, and discreteness of the spectrum. These questions are studied in the greatest detail for one-dimensional Schr\"odinger operators \cite{Glasman, Ismagilov61, Martynov65, Molchanov53, Zelenko67}, with local integrability being a standard condition on the regularity of the potential. Moreover, in last years of growing interest are problems in which the potential is singular and contains delta-functions supported on a discrete set or contains more general Radon measures \cite{AlbeverioKostenkoMalamud10, YanShi14}. Direct generalization of classical theorems to such operators is associated with serious difficulties. These difficulties become greater if the potentials are matrix-valued and the operator acts on vector-valued functions \cite{IsmagilovKostyuchenko07}.

The main purpose of our paper is to ground the fundamental localization principles for the most general operators of the mentioned type. In next papers this will allow us to obtain necessary and/or sufficient constructive conditions for these operators to be semibounded and for their spectrum to be discrete provided that we impose additional restrictions on the matrix potential. The proofs of the results given below is based on the regularization of the differential expression with the help of quasiderivatives \cite{GoriunovMikhailets12UMJ9, GoriunovMikhailetsPankrashkin13, GoriunovMikhailets10MathNotes, Konstantinov15, MikhailetsMolyboga13MFAT2, Molyboga15, SavchukShkalikov99}.

The paper consists of five sections and Appendix. Section~\ref{sect1} is Introduction. Section~\ref{sect2} contains the statement of the problem and formulation of our main results, Theorems \ref{th_1st_loc_princ} and~\ref{th_2nd_loc_princ}. They are  generalizations of the localization principles by Ismagilov \cite{Ismagilov61} to the case of a matrix distributional potential. These theorems are proved in Section~\ref{sect5}. Their proofs are based on the basic Lemma~\ref{lem2} established in Section~\ref{sect4}. Beforehand we will obtain some auxiliary results in Section~\ref{sect3}.

\section{Statement of the problem and main results}\label{sect2}

We consider a linear differential expression
\begin{equation} \label{A}
h(y):=-y''+qy
\end{equation}
in the complex separable Hilbert space $L^2(\mathbb{R},\mathbb{C}^m)$, with $m\geq1$. Here, $y:=(y_1,...,y_m)\in L^2(\mathbb{R}, \mathbb{C}^m)$, and $q:=(q_{i,j})_{i,j=1}^m$ is a matrix potential such that each
\begin{equation*}
q_{i,j}=Q_{i,j}'\quad\mbox{for a certain}\quad Q_{i,j}\in L^{2}_{\mathrm{loc}}(\mathbb{R},\mathbb{C}).
\end{equation*}
Throughout the paper, derivatives are understood in the sense of the theory of distributions. Put $Q:=( Q_{i,j})_{i,j=1}^m$. In the sequel,   the matrix potential $Q$ is supposed to be Hermitian-symmetric, i.e. $Q=Q^*$.

Using the quasiderivatives
\begin{equation*}
y^{[1]}:=y'-Qy\quad\mbox{and}\quad y^{[2]}:=\bigl(y^{[1]}\bigr)'+Qy^{[1]}+Q^2y
\end{equation*}
(see, e.g., \cite{MikhailetsMolyboga13MFAT2}), we write the differential expression \eqref{A} in the form  $h(y)=-y^{[2]}$. Following \cite[Section~1]{MikhailetsMolyboga13MFAT2}, we associate the maximal, preminimal, and minimal operators with this expression in the following way: the maximal operator
\begin{equation} \label{AA}
Hy:=-y^{[2]}
\end{equation}
is defined on the natural widest domain
\begin{equation*}
D(H):=\bigl\{y\in L^2(\mathbb{R},\mathbb{C}^m) :y,y^{[1]}\in \mathrm{AC}_{\mathrm{loc}}(\mathbb{R},\mathbb{C}^m),\; y^{[2]} \in L^2(\mathbb{R}, \mathbb{C}^m)\bigr\}.
\end{equation*}
Here, as usual, $\mathrm{AC}_{\mathrm{loc}}(\mathbb{R},\mathbb{C}^m)$ denotes the set of all vector-valued functions $y:\mathbb{R}\rightarrow\mathbb{C}^m$ that are absolutely continuous on every compact interval $[a,b]\subset\mathbb{R}$. By definition, the preminimal operator $H_0'$ is the restriction of the maximal operator \eqref{AA} to the set of all compactly supported functions $y\in D(H)$, and the minimal operator $H_0$ is the closure of~$H_0'$. It is known \cite[Corollary~2 and Proposition~7]{MikhailetsMolyboga13MFAT2} that the domains of $H$, $H_0'$, and $H_0$ are dense in the Hilbert space $L^2(\mathbb{R},\mathbb{C}^m)$ and that the operators $H_0'$ and $H_0$ are symmetric and
\begin{equation*}
H = (H_0')^*=H_0^*.
\end{equation*}

The main results of the paper are generalizations of the localization principles by Ismagilov \cite{Ismagilov61} to the case of a matrix distributional potential.

Let us introduce some designations. Given a nonempty open set $\Omega\subseteq\mathbb{R}$, we put
\begin{equation*}
\lambda(\Omega):=\inf\biggl\{\frac{\langle H_0y,y\rangle}{\langle y,y\rangle}:\; y\in D(H_0')\setminus\lbrace 0\rbrace,\; \mathrm{supp}\,y\subset \Omega\biggr\}.
\end{equation*}
Here and below, $\langle \cdot,\cdot\rangle$ is the inner product in the Hilbert space $L^2(\mathbb{R},\mathbb{C}^m)$. Since the operator $H_0'$ is symmetric, the inclusion $\langle H_0'y,y\rangle\in\mathbb{R}$ holds;   therefore $\lambda(\Omega)$ is well defined.

We choose a number $\ell>0$ arbitrarily and put
$$
\omega_n^\ell:=\Bigl(\frac{n\ell}{2},\frac{n\ell}{2}+\ell\Bigr) \quad \mbox{for every}\quad n\in\mathbb{Z}.
$$

As in the case of a locally integrable scalar potential, each number
$\lambda(\omega_n^\ell)$ coincides with the smallest eigenvalue of the bounded below selfadjoint operator $H^{\mathrm{D}}(\omega_n^\ell)$ generated by the differential expression \eqref{A} and the homogeneous boundary conditions $y(n\ell/{2})=y(n\ell/{2}+\ell)=0$ in the Hilbert space $L^2(\omega_n^\ell,\mathbb{C}^m)$ (see Appendix below). Therefore these numbers make physical sense.

\begin{theorem}[the first localization principle] \label{th_1st_loc_princ}
The minimal operator $H_0$ is bounded below and selfadjoint if and only if the sequence of numbers  $(\lambda(\omega_n^\ell))_{n=-\infty}^{+\infty}$ is bounded below.
\end{theorem}

\begin{theorem}[the second localization principle] \label{th_2nd_loc_princ}
The operator $H_0$ is a bounded below selfadjoint operator with discrete spectrum if and only if
\begin{equation}\label{T2}
\lambda(\omega_n^\ell)\rightarrow+\infty \quad \mbox{as} \quad |n|\rightarrow\infty.
\end{equation}
\end{theorem}

\begin{remark}
It follows from Theorems \ref{th_1st_loc_princ} and \ref{th_2nd_loc_princ} that if for a certain $\ell>0$ the sequence $(\lambda(\omega_n^\ell))_{n=-\infty}^{+\infty}$ is bounded below or satisfies \eqref{T2}, this sequence will have the same property for every $\ell>0$.
\end{remark}

\begin{remark}
In Theorems \ref{th_1st_loc_princ} and \ref{th_2nd_loc_princ} it is possible to replace all the intervals $\omega_n^\ell$, where
$n\in\mathbb{Z}$, with their shifts at an arbitrarily chosen number $a\in\mathbb{R}$.
\end{remark}

\begin{remark}
Analogs of Theorems \ref{th_1st_loc_princ} and \ref{th_2nd_loc_princ} are true in the case where the differential expression \eqref{A} is given on a semiaxis.
\end{remark}

\begin{remark}
Theorem \ref{th_1st_loc_princ} somewhat generalizes and together with Lemma~\ref{lem2} for $\Omega:=\mathbb{R}$ supplements the known statement \cite{MikhailetsMolyboga13MFAT2} about the self-adjointness of the bounded below operator $H_0$. Specifically, it follows from Theorem~\ref{th_1st_loc_princ} that the operator $H_0$ with the periodic matrix potential $Q$ is bounded below and selfadjoint; cf. \cite{MikhailetsMolyboga08MFAT2}, where the case of $m=1$ is examined.
\end{remark}

\section{Auxiliary results}\label{sect3}

Given vector-valued functions $y,z:\mathbb{R}\to\mathbb{C}^m$, we let $(y,z)$ denote the scalar complex-valued function defined by the formula $(y,z):=y_1\overline{z_1}+\ldots+y_m\overline{z_m}$ on~$\mathbb{R}$. Note that
$$
\langle y,z\rangle=\int\limits_{-\infty}^{\infty}(y,z)dx
$$
if $y,z\in L^2(\mathbb{R},\mathbb{C}^{m})$. Throughout the paper all integrals are understood in the sense of Lebesgue, and $dx$ denotes the Lebesgue measure on $\mathbb{R}$, we omitting the argument $x$ of functions under the integral sign.

We choose a real-valued function $\theta\in C^{\infty}(\mathbb{R})$ such that $\mathrm{supp}\,\theta =[0,\ell]$ and
\begin{gather}\label{f1}
\theta^2(x)+\theta^2(x-\ell/2)=1 \quad\mbox{for every} \quad x\in [\ell/2,\ell].
\end{gather}
An example of this function will be given at the end of the present section. Given $k\in\mathbb{Z}$ and $y\in D(H_0')$, we introduce the functions
\begin{equation*}
\theta_k(x):=\theta(x-k\ell/2),\quad
u_k(x):=\theta^2_k(x)y(x), \quad\mbox{and}\quad
v_k(x):=\theta_k(x)\theta_{k+1}(x)y(x)
\end{equation*}
of $x\in\mathbb{R}$

\begin{lemma}\label{lem1}
Let $y\in D(H_0')$. Then $u_k, v_k\in D(H_0')$ for every $k\in\mathbb{Z}$, and we have the equality
\begin{equation}\label{l1_main}
\begin{aligned}
\langle H_0y,y\rangle=\sum\limits_{k=-\infty}^\infty \langle H_0u_k,u_k\rangle&+2\sum\limits_{k=-\infty}^\infty \langle H_0 v_k,v_k\rangle\\
&-2\sum\limits_{k=-\infty}^\infty \:\int\limits_{k\ell/2+\ell/2}^{k\ell/2+\ell} (\theta'_k\theta_{k+1}-\theta_k\theta'_{k+1})^2\,(y,y)\,dx.
\end{aligned}
\end{equation}
\end{lemma}

\begin{proof}
We choose $k\in\mathbb{Z}$ arbitrarily and will show that $u_k,v_k\in D(H_0')$. Since the functions $y\in\mathrm{AC_{loc}}(\mathbb{R},\mathbb{C}^{m})$ and $\theta_k,\theta_{k+1}\in C^{\infty}(\mathbb{R},\mathbb{R})$ are compactly supported, the functions $u_k$ and $v_k$ are also compactly supported and belong to both the spaces $\mathrm{AC_{loc}}(\mathbb{R}, \mathbb{C}^{m})$ and $L^2(\mathbb{R}, \mathbb{C}^{m})$. Besides,
\begin{align*}
u_k^{[1]}=u_k'-Qu_k=(\theta_k^2 y)' - Q\theta_k^2y&=(\theta_k^2)'y+\theta_k^2 y'-Q\theta_k^2y\\
&=(\theta_k^2)'y+\theta_k^2y^{[1]}\in \mathrm{AC_{loc}}(\mathbb{R},\mathbb{C}^m)
\end{align*}
and
\begin{align*}
u_k^{[2]}&=\bigl(u_k^{[1]}\bigr)' + Qu_k^{[1]} + Q^2u_k =
\bigl((\theta_k^2)'y+\theta_k^2y^{[1]}\bigr)'+
Q\bigl((\theta_k^2)'y+\theta_k^2y^{[1]}\bigr)+Q^2\theta_k^2y\\
&=(\theta_k^2)''y + (\theta_k^2)'y'  + (\theta_k^2)'y^{[1]}+ \theta_k^2\bigl(y^{[1]}\bigr)' + (\theta_k^2)'Qy + \theta_k^2Qy^{[1]} + \theta_k^2Q^2y\\
&= \theta_k^2y^{[2]} + (\theta_k^2)'\bigl(y'+ y^{[1]}+ Qy\bigr) + (\theta_k^2)''y\\
&= \theta_k^2y^{[2]} + 2(\theta_k^2)'\bigl(y^{[1]} + Qy\bigr)+(\theta_k^2)''y\in L_2(\mathbb{R},\mathbb{C}^m).
\end{align*}
Here, we use the fact that $y^{[1]}\in \mathrm{AC_{loc}}(\mathbb{R}, \mathbb{C}^{m})$, $y^{[2]}\in L^2(\mathbb{R}, \mathbb{C}^{m})$, and $Q\in L^2_{\mathrm{loc}}(\mathbb{R}, \mathbb{C}^{m\times m})$. Replacing $u_k$ with $v_k$ and $\theta_k^2$ with $\theta_k\theta_{k+1}$ in the above equalities, we obtain the inclusions $v_k^{[1]}\in \mathrm{AC_{loc}}(\mathbb{R}, \mathbb{C}^{m})$ and $v_k^{[2]}\in L_2(\mathbb{R},\mathbb{C}^m)$. Thus, $u_k, v_k\in D(H_0')$ by the definition of $D(H_0')$.

Let us now prove equality \eqref{l1_main}. Integrating by parts, we write
\begin{equation}\label{l1_1}
\begin{aligned}
\langle H_0y,y\rangle=\langle y^{[2]},y\rangle
&=-\int\limits_{-\infty}^{\infty}\bigl( (y^{[1]})' +Q(y'-Qy)+Q^2y,y\bigr)dx\\
&=\int\limits_{-\infty}^{\infty}\bigl(
(y^{[1]},y')-(Qy',y)
\bigr)dx.
\end{aligned}
\end{equation}
Equality \eqref{l1_1} holds true for every function $y\in D(H_0')$. Since $u_k,v_k\in D(H_0')$, we may put $y:=u_k$ or $y:=v_k$ in this equality and write
\begin{gather}\label{l1_3}
\langle H_0u_k,u_k\rangle=\int\limits_{k\ell/2}^{k\ell/2+\ell}
\bigl((u_k^{[1]},u_k')-(Qu_k',u_k)\bigr)dx,\\
\langle H_0v_k,v_k\rangle=\int\limits_{k\ell/2+\ell/2}^{k\ell/2+\ell}
\bigl((v_k^{[1]},v_k')-(Qv_k',v_k)\bigr)dx. \label{l1_4}
\end{gather}
Here, we use the fact that $\mathrm{supp}\,u_k\subseteq[k\ell/2,k\ell/2+\ell]$ and $\mathrm{supp}\,v_k\subseteq [k\ell/2+\ell/2,k\ell/2+\ell]$.

Owing to \eqref{l1_3} we write
\begin{align*}
&\sum\limits_{k=-\infty}^{\infty}\langle H_0u_k,u_k\rangle\\
=&\sum\limits_{k=-\infty}^{\infty}\int\limits_{k\ell/2}^{k\ell/2+\ell/2}
\bigl((u_k^{[1]},u_k')-(Qu_k',u_k)
\bigr)dx\\
&\quad+\sum\limits_{k=-\infty}^{\infty}\,
\int\limits_{k\ell/2+\ell/2}^{k\ell/2+\ell}
\bigl((u_k^{[1]},u_k')-(Qu_k',u_k)\bigr)dx\\
=&\sum\limits_{j=-\infty}^{\infty}\,
\int\limits_{j\ell/2+\ell/2}^{j\ell/2+\ell}
\bigl((u_{j+1}^{[1]},u_{j+1}')-(Qu_{j+1}',u_{j+1})\bigr)dx\\
&\quad+\sum\limits_{k=-\infty}^{\infty}\,
\int\limits_{k\ell/2+\ell/2}^{k\ell/2+\ell}
\bigl((u_k^{[1]},u_k')-(Qu_k',u_k)\bigr)dx\\
=&\sum\limits_{k=-\infty}^{\infty}\,
\int\limits_{k\ell/2+\ell/2}^{k\ell/2+\ell}
\bigl((u_{k}^{[1]},u_{k}')+(u_{k+1}^{[1]},u_{k+1}')-(Qu_k',u_k)
-(Qu_{k+1}',u_{k+1})\bigr)dx.
\end{align*}
By virtue of this formula and \eqref{l1_4}, we obtain the equalities
\begin{align}
&\sum\limits_{k=-\infty}^\infty \langle H_0u_k,u_k\rangle+2\sum\limits_{k=-\infty}^\infty \langle H_0v_k,v_k\rangle\notag\\
=&\sum\limits_{k=-\infty}^{\infty}\,
\int\limits_{k\ell/2+\ell/2}^{k\ell/2+\ell}
\bigl((u_{k}^{[1]},u_{k}')+(u_{k+1}^{[1]},u_{k+1}')+2(v_{k}^{[1]},v_{k}')
\notag\\
&\qquad\qquad\quad
-(Qu_k',u_k)-(Qu_{k+1}',u_{k+1})-2(Qv_k',v_k)\bigr)dx\notag\\
=&\sum\limits_{k=-\infty}^\infty\,
\int\limits_{k\ell/2+\ell/2}^{k\ell/2+\ell}\bigl((u_k',u_k')+ 2(v_k',v_k') +(u_{k+1}',u_{k+1}')\notag\\
&\qquad\qquad\quad
-(Qu_k,u_k')-2(Qv_k,v_k')-(Qu_{k+1},u_{k+1}')\notag\\
&\qquad\qquad\quad
-(Qu_k',u_k)-2(Qv_k',v_k)-
(Qu_{k+1}',u_{k+1})\bigr)
dx.\label{l1_5}
\end{align}

Let us show that for every $k\in\mathbb{Z}$ the last integrand is equal to
\begin{equation}\label{l1_uk}
(y',y')+2(\theta_k'\theta_{k+1}-\theta_k\theta'_{k+1})^2(y,y)-
(Qy,y')-(Qy',y).
\end{equation}
We note beforehand that
\begin{align*}
(u_k',u_k')&=\bigl((\theta_k^2y)',(\theta_k^2y)'\bigr)=
\bigl((2\theta_k\theta_k'y+\theta_k^2y'),
(2\theta_k\theta_k'y+\theta_k^2y')\bigr)\\
&=4\theta_k^2(\theta_k')^2(y,y)+2\theta_k^3\theta_k'\bigl((y,y')+
(y',y)\bigr)+\theta_k^4(y',y')
\end{align*}
and
\begin{align*}
(v_k',v_k')&= \bigl((\theta_k\theta_{k+1}y)',(\theta_k\theta_{k+1}y)'\bigr)\\
&=\bigl((\theta_{k}\theta_{k+1})'y+
\theta_{k}\theta_{k+1}y',(\theta_{k}\theta_{k+1})'y+
\theta_{k}\theta_{k+1}y'\bigr)\\
&=((\theta_{k}\theta_{k+1})')^2(y,y)+
\theta_k\theta_{k+1}(\theta_k\theta_{k+1})'\bigl((y,y')+(y',y)\bigr)\\
&\quad+\theta_k^2\theta_{k+1}^2(y',y').
\end{align*}
It follows directly from formula \eqref{f1} and the definition of $\theta_k$ that
\begin{equation}\label{m_main}
\begin{gathered}
\theta^2_k(x) + \theta^2_{k+1}(x) = 1 \quad \mbox{and}\quad \theta_k(x)\theta_k'(x) + \theta_{k+1}(x)\theta_{k+1}'(x) = 0 \\
\mbox{for every}\quad x\in[k\ell/2+\ell/2,k\ell/2+\ell].
\end{gathered}
\end{equation}
Therefore we have the following equalities on $[k\ell/2+\ell/2,k\ell/2+\ell]$:
\begin{align*}
&(u_k',u_k') + 2(v_k',v_k') + (u_{k+1}',u_{k+1}')\\
=&2(y,y)\bigl(2\theta_k^2(\theta_k')^2+
((\theta_{k}\theta_{k+1})')^2+2\theta_{k+1}^2(\theta_{k+1}')^2\bigr)\\
&+2\bigl((y,y')+(y',y)\bigr)(\theta_k^3\theta_k'+
\theta_k\theta_{k+1}(\theta_k\theta_{k+1})'+\theta_{k+1}^3\theta_{k+1}')\\
&+(y',y')(\theta_k^4+ 2\theta_k^2\theta_{k+1}^2+ \theta_{k+1}^4)\\
=&2(y,y)\bigl(2\theta_k^2(\theta_k')^2+ (\theta_k')^2\theta_{k+1}^2 + 2\theta_k\theta_k'\theta_{k+1}\theta_{k+1}' + \theta_k^2(\theta_{k+1}')^2 + 2\theta_{k+1}^2(\theta_{k+1}')^2\bigr)\\
&+2\bigl((y,y')+(y',y)\bigr)(\theta_k^3\theta_k' + \theta_k\theta_k'\theta_{k+1}^2 + \theta_k^2\theta_{k+1}\theta_{k+1}' + \theta_{k+1}^3\theta_{k+1}')\\
&+(y',y')(\theta_k^2+\theta_{k+1}^2)^2\\
=&2(y,y)\bigl(2\theta_k^2(\theta_k')^2 + 4\theta_k\theta_k'\theta_{k+1}\theta_{k+1}'+
2\theta_{k+1}^2(\theta_{k+1}')^2\\
&\quad+ (\theta_k')^2\theta_{k+1}^2 -2\theta_k\theta_k'\theta_{k+1}\theta_{k+1}' + \theta_k^2(\theta_{k+1}')^2\bigr)\\
&+2\bigl((y,y')+(y',y)\bigr)\bigl(\theta_k^2(\theta_k\theta_k' + \theta_{k+1}\theta_{k+1}') + \theta_{k+1}^2(\theta_k\theta_k'+\theta_{k+1}\theta_{k+1}')\bigr)+(y',y')\\
=&2(y,y)\bigl(2(\theta_k\theta_k'+\theta_{k+1}\theta_{k+1}')^2)+
(\theta_k'\theta_{k+1}-\theta_k\theta_{k+1}')^2\bigr)+(y',y')\\
=&2(y,y)(\theta_k'\theta_{k+1}-\theta_k\theta_{k+1}')^2+(y',y');
\end{align*}
i.e.,
\begin{equation}\label{l1_fb}
(u_k',u_k')+2(v_k',v_k')+(u_{k+1}',u_{k+1}')=
(y',y')+2(\theta_k'\theta_{k+1}-\theta_k\theta_{k+1}')^2(y,y).
\end{equation}
Besides,
\begin{equation*}\label{l1_sec1}
(Qu_k,u_k')=\bigl(Q(\theta_k^2y),(\theta_k^2y)'\bigr)=
\theta_k^2(Qy,2\theta_k\theta_k'y+\theta_k^2y')
=2\theta_k^3\theta_k'(Qy,y)+\theta_k^4 (Qy,y')
\end{equation*}
and
\begin{align*}
(Qv_k,v_k')&=\bigl(Q(\theta_k\theta_{k+1}y),(\theta_k\theta_{k+1}y)'\bigr)=
\theta_k\theta_{k+1}\bigl(Qy,(\theta_k\theta_{k+1})'y+
\theta_k\theta_{k+1}y'\bigr)\\
&=\theta_k\theta_{k+1}(\theta_k\theta_{k+1})'(Qy,y)+
\theta_k^2\theta^2_{k+1}(Qy,y').
\end{align*}

Hence, in view of \eqref{m_main}, we have the following equalities on the compact interval $[k\ell/2+\ell/2,k\ell/2+\ell]$:
\begin{align*}
&(Qu_k,u_k') +2(Qv_k,v_k')+(Qu_{k+1},u_{k+1}')\\
=&2(Qy,y)(\theta_k^3\theta_k' +\theta_k\theta_k'\theta_{k+1}^2 + \theta_k^2\theta_{k+1}\theta_{k+1}'  +\theta_{k+1}^3\theta_{k+1}')\\
&+(Qy,y')(\theta_k^4+2\theta_k^2\theta_{k+1}^2+\theta_{k+1}^4)\\
=&2(Qy,y)\bigl(\theta_k^2(\theta_k\theta_k'+\theta_{k+1}\theta_{k+1}') +\theta_{k+1}^2(\theta_k\theta_k' +\theta_{k+1}\theta_{k+1}')\bigr)
+(Qy,y')(\theta_k^2+\theta_{k+1}^2)^2\\
=&(Qy,y');
\end{align*}
i.e.,
\begin{equation}\label{l1_sb}
(Qu_k,u_k') +2(Qv_k,v_k')+(Qu_{k+1},u_{k+1}')=(Qy,y').
\end{equation}
Since $Q=Q^*$, the equalities
\begin{equation}\label{l2_sb}
\begin{aligned}
(Qu_k',u_k) +2(Qv_k',v_k)+(Qu_{k+1}',u_{k+1})
=&(u_k',Qu_k) +2(v_k',Qv_k)+(u_{k+1}',Qu_{k+1})\\
=&\overline{(Qy,y')}=(Qy',y)
\end{aligned}
\end{equation}
hold on the same interval.

Owing to \eqref{l1_fb}--\eqref{l2_sb} we conclude that the last integrand in \eqref{l1_5} equals \eqref{l1_uk} for every $k\in\mathbb{Z}$. Hence, according to \eqref{l1_5} and \eqref{l1_1}, we have the equalities
\begin{align*}
&\sum\limits_{k=-\infty}^\infty \langle H_0u_k,u_k\rangle+2\sum\limits_{k=-\infty}^\infty \langle H_0v_k,v_k\rangle\\
=&\sum\limits_{k=-\infty}^{\infty}\,
\int\limits_{k\ell/2+\ell/2}^{k\ell/2+\ell}
\bigl((y',y')+2(\theta_k'\theta_{k+1}-\theta_k\theta'_{k+1})^2(y,y)-
(Qy,y')-(Qy',y)\bigr)dx\\
=&\int\limits_{-\infty}^{\infty}\bigl((y'-Qy,y')-(Qy',y)\bigr)dx+
2\sum\limits_{k=-\infty}^{\infty}\,
\int\limits_{k\ell/2+\ell/2}^{k\ell/2+\ell} (\theta_k'\theta_{k+1}-\theta_k\theta'_{k+1})^2\,(y,y)\, dx\\
=&\langle H_0y,y\rangle+2\sum\limits_{k=-\infty}^{\infty}\,
\int\limits_{k\ell/2+\ell/2}^{\ell/2+\ell} (\theta_k'\theta_{k+1}-\theta_k\theta'_{k+1})^2\,(y,y)\,dx.
\end{align*}
This immediately implies the required formula \eqref{l1_main}.
\end{proof}

\textbf{Example.} Let us give an example of a real-valued function
$\theta \in C^{\infty}(\mathbb{R})$ that satisfies the equality $\mathrm{supp}\,\theta=[0,\ell]$ and condition \eqref{f1}. Recall that this function is used in Lemma~\ref{lem1}. We choose a function  $\eta_0\in C^{\infty}(\mathbb{R})$ such that $\mathrm{supp}\,\eta_0=[0,\ell]$ and $\eta_0(x)>0$ for every $x\in (0,\ell)$. Let $\eta$ denote the $\ell$-periodic extension of the function $\eta_0^2$ over the whole~$\mathbb{R}$. We introduce the real-valued function
\begin{equation*}
\theta(x):=\frac{\eta_0(x)}{\sqrt{\eta(x)+\eta(x-h/2)}} \quad \mbox{of}\quad x\in\mathbb{R}.
\end{equation*}
Since $\eta(x)+\eta(x-h/2)>0$ for every $x\in\mathbb{R}$, this function is well defined and satisfies the conditions $\theta\in C^{\infty}(\mathbb{R})$ and $\mathrm{supp}\,\theta=[0,\ell]$. It also satisfies condition \eqref{f1}. Indeed, given $x\in[\ell/2,\ell]$, we obtain the equalities
\begin{equation*}
\begin{gathered}
\theta^2(x)+\theta^2(x-\ell/2)=
\frac{\eta^2_0(x)}{\eta(x)+\eta(x-\ell/2)}+
\frac{\eta^2_0(x-\ell/2)}{\eta(x-\ell/2)+\eta(x-\ell)}=\\
=\frac{\eta^2_0(x)+\eta^2_0(x-\ell/2)}{\eta(x)+\eta(x-\ell/2)}=1
\end{gathered}
\end{equation*}
in view of the definition of $\eta$.

\section{Basic Lemma}\label{sect4}

We put
\begin{equation*}
\varkappa:=\max\bigl\{|\theta'(x)|:0\leq x\leq\ell\bigr\}.
\end{equation*}

\begin{lemma}\label{lem2}
Let $\Omega$ be an nonempty open subset of $\mathbb{R}$. Then there exists $n\in\mathbb{Z}$ such that $\omega_n^\ell\cap\Omega\neq\varnothing$ and
\begin{equation}\label{l2_main}
\nu(\omega_n^\ell)\leq\lambda(\Omega)+8\varkappa^2.
\end{equation}
\end{lemma}

\begin{proof}
It follows from property \eqref{f1} and the definition of $\varkappa$ that $0\leq\theta_k(x)\leq 1$ and $|\theta_k'(x)|\leq \varkappa$ for arbitrary $k\in\mathbb{Z}$ and $x\in\mathbb{R}$. Therefore
$(\theta_k'\theta_{k+1}-\theta_k\theta'_{k+1})^2(x)\leq 4\varkappa^2$;
hence, owing to Lemma~\ref{lem1}, the inequality
\begin{equation}\label{l2-g}
\langle H_0y,y\rangle\geq\sum\limits_{k=-\infty}^\infty \langle H_0u_k,u_k\rangle+2\sum\limits_{k=-\infty}^\infty \langle H_0v_k,v_k\rangle-8\varkappa^2\langle y,y\rangle
\end{equation}
holds true for every $y\in D(H_0')$.

Note that the equality
\begin{equation}\label{l2-h}
\langle y,y\rangle=\sum\limits_{k=-\infty}^\infty\langle u_k,u_k\rangle+2\sum\limits_{k=-\infty}^\infty\langle v_k,v_k\rangle
\end{equation}
is valid for every $y\in D(H_0')$. Indeed,
\begin{align*}
\sum\limits_{k=-\infty}^{\infty}\langle u_k,u_k\rangle&=
\sum\limits_{k=-\infty}^{\infty}\langle \theta_k^2y,\theta_k^2y\rangle=
\sum\limits_{k=-\infty}^{\infty}\int\limits_{k\ell/2}^{k\ell/2+\ell}
(\theta_k^2y,\theta_k^2y)dx\\
&=\sum\limits_{k=-\infty}^{\infty}\int\limits_{k\ell/2}^{k\ell/2+\ell/2}
\theta_k^4(y,y)dx+\sum\limits_{k=-\infty}^{\infty}\,
\int\limits_{k\ell/2+\ell/2}^{k\ell/2+\ell}\theta_k^4(y,y)dx\\
&=\sum\limits_{j=-\infty}^{\infty}\,
\int\limits_{j\ell/2+\ell/2}^{j\ell/2+\ell}\theta_{j+1}^4(y,y)dx+
\sum\limits_{k=-\infty}^{\infty}\,
\int\limits_{k\ell/2+\ell/2}^{\ell/2+\ell}\theta_k^4(y,y)dx\\
&=\sum\limits_{k=-\infty}^{\infty}\,
\int\limits_{k\ell/2+\ell/2}^{k\ell/2+\ell}
(\theta_{k+1}^4+\theta_{k}^4)(y,y)dx.
\end{align*}
Besides,
\begin{align*}
\sum\limits_{k=-\infty}^{\infty}\langle v_k,v_k\rangle&=
\sum\limits_{k=-\infty}^{\infty}\langle \theta_k\theta_{k+1}y,\theta_k\theta_{k+1}y\rangle\\
&=\sum\limits_{k=-\infty}^{\infty}\,
\int\limits_{k\ell/2+\ell/2}^{\ell/2+\ell}
(\theta_k\theta_{k+1}y,\theta_k\theta_{k+1}y)dx\\
&=\sum\limits_{k=-\infty}^{\infty}\,
\int\limits_{k\ell/2+\ell/2}^{k\ell/2+\ell}
\theta_k^2\theta^2_{k+1}( y,y)dx.
\end{align*}
Therefore
\begin{align*}
\sum\limits_{k=-\infty}^{\infty}\langle u_k,u_k\rangle+2\sum\limits_{k=-\infty}^{\infty}\langle v_k,v_k\rangle
&=\sum\limits_{k=-\infty}^{\infty}\,
\int\limits_{k\ell/2+\ell/2}^{k\ell/2+\ell}
(\theta_{k}^2+\theta_{k+1}^2)^2\,(y,y)dx\\
&=\sum\limits_{k=-\infty}^{\infty}\,
\int\limits_{k\ell/2+\ell/2}^{k\ell/2+\ell}(y,y)dx=\langle y,y\rangle.
\end{align*}

It follows from the definition of $\lambda(\Omega)$ that for every number $\delta>0$ there exists a function $y\in D(H_0')$ such that
$\mathrm{supp}\,y\subset\Omega$ and
\begin{equation*}\label{l2-j}
\langle H_0y,y\rangle<(\lambda(\Omega)+\delta)\langle y,y\rangle.
\end{equation*}
Applying \eqref{l2-g} and \eqref{l2-h} to the last formula, we obtain the inequality
\begin{align*}
\sum\limits_{k=-\infty}^\infty \langle H_0u_k,u_k\rangle&+2\sum\limits_{k=-\infty}^\infty \langle H_0v_k,v_k\rangle\\
&-(8\varkappa^{2}+\lambda(\Omega)+\delta)
\biggl(\,\sum\limits_{k=-\infty}^\infty\langle u_k,u_k\rangle+2\sum\limits_{k=-\infty}^\infty\langle v_k,v_k\rangle \biggr)<0.
\end{align*}
Grouping summands, we write this inequality in the form
\begin{equation*}\label{l2-f}
\begin{aligned}
&\sum\limits_{k=-\infty}^\infty\bigl(\langle H_0u_k,u_k\rangle-(\lambda(\Omega)+8\varkappa^{2}+\delta)\langle u_k,u_k\rangle\bigr)\\
+2&\sum\limits_{k=-\infty}^\infty\bigl(\langle H_0v_k,v_k\rangle-(\lambda(\Omega)+8\varkappa^{2}+\delta)\langle v_k,v_k\rangle\bigr)<0.
\end{aligned}
\end{equation*}
Here, at least one of the summands is less than zero. Let a negative summand have an index $k=k_0$. Then $u_{k_0}\not\equiv 0$ or $v_{k_0}\not\equiv 0$ for otherwise this summand would equal to zero. Hence,
\begin{equation*}
\emptyset\neq (\mathrm{supp}\,u_{k_0})\cup (\mathrm{supp}\,v_{k_0})\subset \Omega \cap \omega_{k_0}^\ell;
\end{equation*}
i.e., $\omega_{k_0}^\ell\cap\Omega\neq\emptyset$. Besides,
\begin{equation*}
\frac{\langle H_0w_{k_0},w_{k_0}\rangle}{\langle w_{k_0},w_{k_0}\rangle}<\lambda(\Omega)+8\varkappa^{2}+\delta,
\end{equation*}
with $w_{k_0}:=u_{k_0}$ or $w_{k_0}:=v_{k_0}$. It follows from this inequality and the inclusion $\mathrm{supp}\,w_{k_0}\subset\omega_{k_0}^\ell$ that
\begin{equation*}
\lambda(\omega_{k_0}^\ell)<\lambda(\Omega)+8\varkappa^{2}+\delta.
\end{equation*}
Passing here to the limit as $\delta\to 0+$, we obtain the required inequality \eqref{l2_main}.
\end{proof}

\section{Proofs of the main results}\label{sect5}

We will prove Theorems \ref{th_1st_loc_princ} and \ref{th_2nd_loc_princ} with the help of Lemma~\ref{lem2}.

\begin{proof}[Proof of Theorem $\ref{th_1st_loc_princ}$]
\emph{Sufficiency.} Assume that there exists a number  $\alpha\in\mathbb{R}$ such that $\lambda(\omega_n^\ell)\geq\alpha$ for every $n\in\mathbb{Z}$. Then, according to Lemma~\ref{lem2} for $\Omega:=\mathbb{R}$, we have the inequalities
\begin{equation*}
\lambda(\mathbb{R})\geq\lambda(\omega_n^\ell)-8\varkappa^{2}\geq \alpha-8\varkappa^{2}.
\end{equation*}
Hence, $H_0'\geq(\alpha-8\varkappa^{2})I$, where $I$ is the identity operator. Then the operator $H_0$ is also bounded below so that it is selfadjoint due to \cite[Corollary~2]{MikhailetsMolyboga13MFAT2}. Sufficiency is proved.

\emph{Necessity} is obvious. Indeed, if $H_0\geq\beta I$ for certain $\beta\in\mathbb{R}$, then $\nu(\omega_n^\ell)\geq\beta$ for every  $n\in\mathbb{Z}$.
\end{proof}

\begin{proof}[Proof of Theorem $\ref{th_2nd_loc_princ}$]
\emph{Sufficiency.} Assume that $\lambda(\omega_n^\ell)\to+\infty$ as $|n|\rightarrow\infty$. Then, owing to Theorem~\ref{th_1st_loc_princ}, the operator $H_0$ is bounded below and selfadjoint; hence, $H_0=H$. Let us prove that its spectrum is discrete, i.e. $\sigma_{\mathrm{ess}}(H_0)=\emptyset$. We arbitrarily choose a number $r>0$. By our assumption, there exists a number $n_r\in\mathbb{N}$ such that
\begin{equation*}
\lambda(\omega_n^\ell)>r\quad\mbox{whenever}\quad|n|\geq n_r.
\end{equation*}
Let us use Lemma \ref{lem2} for the open set
\begin{equation*}
\Omega:=\Bigl(-\infty,-\frac{n_r\ell}{2}-\ell\Bigr)
\cup\Bigl(\frac{n_r\ell}{2}+\ell,\infty\Bigr).
\end{equation*}
Observe that $\omega_n^\ell\cap\Omega\neq\emptyset\Rightarrow|n|\geq n_r$. Therefore it follows from this lemma that
\begin{equation}\label{f18}
\langle H_0y,y\rangle\geq(r-8\varkappa^{2})\langle y,y\rangle
\quad\mbox{whenever}\quad y\in D(H_0')\;\;
\mbox{and}\;\;\mathrm{supp}\,y\subset\Omega.
\end{equation}

We put $\gamma:=n_r\ell/2+2\ell$ and consider the decomposition of the Hilbert space $L^2(\mathbb{R},\mathbb{C}^m)$ in the orthogonal sum of its subspaces
\begin{equation*}
L^2(\mathbb{R})=L^2(-\infty, -\gamma]\oplus L^2[-\gamma,\gamma]\oplus L^2[\gamma, \infty).
\end{equation*}
For the sake of brevity of formulas in the proof, we omit the expression $\mathbb{C}^m$ and exterior parentheses in designations of spaces of vector-valued functions. For example, $L^2(-\infty, -\gamma]$ stands for the space $L^2((-\infty, -\gamma],\mathbb{C}^m)$. Besides, we identify vector-valued functions given on an interval $G\subset\mathbb{R}$ with their extensions by zero over the whole~$\mathbb{R}$. In this sense, $L^2(G)$ is considered as a subspace of $L^2(\mathbb{R})$. With the operator $H_0$ and this decomposition we associate three unbounded operators $H_jy:=-y^{[2]}$, where $j\in\{1,2,3\}$. They are defined respectively on the linear manifolds
\begin{gather*}
D(H_1):=\bigl\{y\in L^2(-\infty,-\gamma]:y,y^{[1]}\in \mathrm{AC_{loc}}(-\infty,-\gamma],\\ y(-\gamma)=y^{[1]}(-\gamma)=0,\,y^{[2]}\in L^2(-\infty,-\gamma] \bigr\},\\
D(H_2):=\bigl\{y\in L^2[-\gamma,\gamma]:y,y^{[1]}\in \mathrm{AC}[-\gamma,\gamma],\\ y(-\gamma)=y^{[1]}(-\gamma)=y(\gamma)=y^{[1]}(\gamma)=0,\,y^{[2]}\in L^2[-\gamma,\gamma] \bigr\},\\
D(H_3):=\bigl\{y\in L^2[\gamma,\infty):y,y^{[1]}\in \mathrm{AC_{loc}}[\gamma,\infty),\\y(\gamma)=y^{[1]}(\gamma)=0,\,
y^{[2]}\in L^2[\gamma,\infty) \bigr\}.
\end{gather*}
Each operator $H_j$ is closed and a restriction of $H_0$. This follows from the fact that $(O_jy)^{[1]}=O_j(y^{[1]})$ and $(O_jy)^{[2]}=O_j(y^{[2]})$ for every $y\in \nobreak D(H_j)$, where $O_j$ is the operator of the extension of a function by zero from the corresponding set onto the whole $\mathbb{R}$. Hence, $H_0$ is an extension of the orthogonal sum $H_1\oplus H_2\oplus H_3$ of these operators. Since $H_0$ is bounded below, all $H_1$, $H_2$, $H_3$ are also bounded below.

We let $H_j^\mathrm{F}$ denote the selfadjoint Friedrichs extension of the semibounded operator $H_j$, with $j\in\{1,2,3\}$. The spectrum of $H_2^\mathrm{F}$ is discrete \cite{SavchukShkalikov99}. Owing to property \eqref{f18} and the definition of $\gamma$, the operators  $H_1^\mathrm{F}$ and $H_3^\mathrm{F}$ are bounded below by the number $r-8\varkappa^{2}$.  The resolvents of the selfadjoint operators $H_0$ and $H_1^\mathrm{F}\oplus H_2^\mathrm{F}\oplus H_3^\mathrm{F}$ in $L^2(\mathbb{R})$ differ in an operator with finite rank. Hence,
\begin{equation*}
\sigma_{\mathrm{ess}}(H_0) = \sigma_{\mathrm{ess}}(H_1^\mathrm{F})\cup
\sigma_{\mathrm{ess}}(H_3^\mathrm{F}),
\end{equation*}
which yields the equality
\begin{equation*}
\sigma_{\mathrm{ess}}(H_0)\cap(-\infty,r-8\varkappa^{2})=\emptyset.
\end{equation*}
Thus, $\sigma_{\mathrm{ess}}(H_0)=\emptyset$ because the number $r>0$ is arbitrarily chosen. Sufficiency is proved.

\emph{Necessity.} Assume that $H_0$ is a bounded below selfadjoint operator with discrete spectrum. Let us deduce property \eqref{T2} by means of proof by contradiction. Suppose the contrary, i.e. there exists a number $r>0$ and sequence $(n_k)_{k=1}^{\infty}\subset\mathbb{Z}$ such that $|n_k|\to\infty$ as $k\to\infty$ and that $\lambda(\omega_{n_k}^\ell)<r$. Passing to a subsequence, we may suppose that $\omega_{n_k}^\ell\cap\omega_{n_p}^\ell=\emptyset$ whenever $k\neq p$. It follows from the definition of $\lambda(\omega_{n_k}^\ell)$ that for every integer $k\geq1$ there exists a vector-valued function  $y_k\in D(H_0')\setminus\{0\}$ such that
$\mathrm{supp}\,y_k\subset \omega_{n_k}^\ell$ and
$\langle H_0'y_k - ry_k,y_k\rangle<0$. Let $G$ be a linear span of $\{y_k:1\leq k\in\mathbb{Z}\}$. Since  $\mathrm{supp}\,y_k\cap\mathrm{supp}\,y_p=\emptyset$ whenever $k\neq p$, we deduce the properties $\dim G=\infty$ and
\begin{equation*}
\langle H_0y-ry,y\rangle<0\quad\mbox{for every}\quad
y\in G\setminus\{0\}.
\end{equation*}
Therefore, applying \cite[Chapter~1, Theorem~13]{Glasman} to the  selfadjoint operator $H_0$, we conclude that the set $\sigma(H_0)\cap (-\infty,r)$ is infinite. This contradicts our assumption, according to which the spectrum of $H_0$ is bounded below and does not contain any limit points. Necessity is proved.
\end{proof}

\section*{Appendix}

Let $n\in\mathbb{Z}$. In the complex separable Hilbert space $L^2(\omega_n^\ell,\mathbb{C}^m)$ we consider the operator $H^{\mathrm{D}}(\omega_n^\ell)$ which is the restriction of the maximal operator $H$ on the set of all vector-valued functions $y\in D(H)$ that satisfy the boundary conditions $y(n\ell/2)=0$ and $y(n\ell/2+\ell)=0$.
The operator $H^{\mathrm{D}}(\omega_n^\ell)$ is selfadjoint and bounded below \cite{Konstantinov15}.

\begin{theoremapp}\label{thA}
The operator $H^{\mathrm{D}}(\omega_n^\ell)$ is the Friedrichs extension of the minimal operator $H_{0}(\omega_n^\ell)$ generated by the differential expression \eqref{A} on the interval $\omega_n^\ell$.
\end{theoremapp}

Since the spectrum of $H^{\mathrm{D}}(\omega_n^\ell)$ is discrete \cite{SavchukShkalikov99}, it follows from Theorem~\ref{thA} and properties of the Friedrichs extension that the first eigenvalue of $H^{\mathrm{D}}(\omega_n^\ell)$ coincides with the number $\lambda(\omega_n^\ell)$ introduced in Section~\ref{sect2}.

Let us outline the proof of Theorem~\ref{thA}. This theorem is known in the case where $q=0$ and $m=1$. If $q=0$ and $m\geq 2$, the minimal operator $H_{0}(\omega_n^\ell)$ is the orthogonal sum of $m$ scalar minimal operators. It follows then from the construction of the Friedrichs extension (see, e.g., \cite[Section~124]{RieszSz-Nagy}) that the operator $H^{\mathrm{D}}(\omega_n^\ell)$ is the orthogonal sum of $m$ Friedrichs extensions of the scalar minimal operators. If $q\neq0$ and $m\geq1$, then, using the reasoning from \cite{Molyboga15}, we may show that the operator $H^{\mathrm{D}}(\omega_n^\ell)$ is the form-sum of the free Hamiltonian ($q\equiv0$) and the quadratic form corresponding to the singular potential. The latter is zero relative form-bounded. This implies Theorem~\ref{thA} in view of KLMN theorem.

\end{document}